\documentclass{amsart}
\usepackage{amssymb,amsmath,amsthm,graphics,amscd,amsfonts}
\usepackage{color}

\theoremstyle{plain}
\newtheorem{theorem}{Theorem}[section]
\newtheorem{proposition}[theorem]{Proposition}
\newtheorem{corollary}[theorem]{Corollary}

\newtheorem{example}[theorem]{Example}
\newtheorem{definition}[theorem]{Definition}
\newtheorem{problem}[theorem]{Problem}

\newtheorem{remark}[theorem]{Remark}

\newcommand{\bfo}{{\bf o}}

\newcommand{\bfC}{{\mathbb C}}
\newcommand{\bfP}{{\mathbb P}}
\newcommand{\bfR}{{\mathbb R}}
\newcommand{\bfZ}{{\mathbb Z}}

\newcommand{\bfQ}{{\mathbb Q}}

\newcommand{\barg}{{\overline g}}

\newcommand{\barpartial}{{\overline \partial}}

\newcommand{\mapright}[1]{\smash{\mathop{   \hbox to 0.7cm{\rightarrowfill}}
  \limits^{#1}}}

\newcommand{\Ker}{{\rm Ker}}
\newcommand{\ArrowF}{{\overrightarrow{F}}}

\def\p{\partial}

\def\R{{\mathbb R}}

\begin{document}

\title{Some topics on Ricci solitons and self-similar solutions to mean curvature flow }
\author{Akito Futaki}
\address{Department of Mathematics, Tokyo Institute of Technology, 2-12-1,
O-okayama, Meguro, Tokyo 152-8551, Japan}
\email{futaki@math.titech.ac.jp}
\date{June 4, 2012 }

\begin{abstract} 
In this survey article, we discuss some topics on self-similar solutions 
to the Ricci flow and the mean curvature flow. Self-similar solutions to the Ricci flow 
is known as Ricci solitons.
In the first part of this paper we discuss a lower diameter bound for compact manifolds with
shrinking Ricci solitons. Such a bound can be obtained from an eigenvalue estimate for a twisted
Laplacian, called the Witten-Laplacian. In the second part we discuss self-similar solutions to the mean curvature flow on cone manifolds.
Many results have been obtained for solutions in $\bfR^n$ or $\bfC^n$. We see that many of them
extend to cone manifolds, and in particular results on $\bfC^n$ for special Lagrangians and self-shrinkers 
can be extended to toric Calabi-Yau cones. We also see that a similar lower diameter bound can be obtained for
self-shrinkers to the mean curvature flow as in the case of shrinking Ricci solitons.
\end{abstract}

\maketitle

\section{Introduction}

In this survey paper we discuss self-similar solutions to the two major geometric flows, the Ricci flow and the mean curvature flow.
The self-similar solutions to the Ricci flow are called the Ricci solitons. The self-similar solutions appear as rescaling limits of the 
singularities of the corresponding flows, see \cite{Ham93}, \cite{Sesum06CAG}, \cite{Cao} for the Ricci flow, and \cite{H} for the mean 
curvature flow.  
The two kinds of self-similar solutions have common aspects. A typical such aspect is that there is a common eigenvalue
 $-2\lambda$ for the twisted Laplacian $ \Delta_f = \Delta - \nabla f \cdot$, that is
 $$ \qquad \Delta_fu = \Delta u - g^{ij} \nabla_if\nabla_j u$$
 where 
$$ \mathrm{Ric}(g) - \gamma g + \nabla\nabla f = 0$$
in the case of Ricci flow (see Step 1 in the proof of Theorem  \ref{sol3}), and where 
$$ \vec{H} = - \lambda x^\perp$$
and
$$ f = 2\lambda \left(\frac{|x|^2}4 - \frac n{4\lambda}\right) $$
in the case of mean curvature flow $x : M^n \to \bfR^{n+p}$ (see Theorem \ref{th3}). Combining this with an eigenvalue
estimate for $\Delta_f$ given in Step 2 in the proof of Theorem \ref{sol3} we obtain lower diameter bounds both for the 
Ricci soliton (Theorem \ref{sol3}) 
and the immersed submanifold of the self-similar solution to the mean curvature flow(Theorem \ref{th4}). These two are treated in section 2 and section 6. 

In section 3 we study the mean curvature flow and its self-similar solutions on Riemannian cone manifolds. 
The self-similar solutions of the mean curvature flow have been studied for immersions into $\bfR^{n+p}$ because 
we need to have position vectors and their orthogonal projection to define self-similar solutions. 
The idea of \cite{FHY12}, on which this section is based, is that there are natural position vectors and orthogonal projections for 
immersions into Riemannian cone manifolds. 
We propose
to study the self-similar solutions of the mean curvature flow into Riemannian cone manifolds 
because many earlier works extend to the cone situation. As a typical such result 
we see that
the self-similar solutions are obtained as the limit of parabolic rescalings of the type $\mathrm{I}_c$ singularity,
extending an earlier work of Huisken \cite{H}. 

After a brief introduction to toric Sasaki-Einstein manifolds in section 4,  
we give in section 5 a construction of special Lagrangian submanifolds in toric Calabi-Yau cones,
extending an earlier example  in $\bfC^n$ by Harvey-Lawson \cite{HarveyLawson82Acta}.
This section is based on \cite{FHY12}. 
Yamamoto \cite{Yamamoto} has further constructions of compact Lagrangian self-shrinkers. 

 In section 7 
eternal solutions to the Ricci flow are constructed on certain line bundles over toric Fano manifolds. 
This section is based on \cite{FutakiWang11}.

\vspace{3mm}

\noindent{\bf Acknowledgement}. This paper is partly based on the author's notes of the talk delivered at 
the Modern Mathematics Seminar Series of 
Mathematical Sciences Center of Tsinghua University in September 2011,
and the later discussions with H.-Z. Li and X.-D.Li. 
The author would like to express his gratitude to Professor S.-T. Yau and Prof. Y.-S. Poon for
the invitation to the MSC and hospitality. 

\section{Ricci solitons}

A complete Riemannian metric $g$ on a smooth manifold $M$ is called
a {\it Ricci soliton} if there is a $\gamma \in \bfR$ and a vector field $X$ such that
$$ 2\operatorname{Ric}(g) - 2\gamma g + \mathcal L_X g = 0.$$
If $X = 0$ then $\operatorname{Ric}(g) = \gamma g$, i.e. $g$ is Einstein.
In this case we say $g$ is {\it trivial}.
We say that a nontrivial soliton
is {\it expanding}, {\it steady} or {\it shrinking} according as 
$\gamma < 0$, $\gamma = 0$ or $\gamma > 0$.

If $X = \operatorname{grad} f$ for some $f \in C^\infty(M)$ then 
$$ \operatorname{Ric}(g) - \gamma g + \nabla\nabla f = 0.$$
In this case $g$ is called a {\it gradient soliton}.

Given a Ricci soliton, let $Y_t$ be the time dependent vector field
$$ Y_t := - \frac 1{2\gamma t} X$$
and let $\varphi_t$ be the flow generated by $Y_t$.
If we set
$$ g(t) = -2\gamma t\varphi_t^{\ast}g$$
then $g(t)$ satisfies the Ricci flow equation
$$ \frac{\partial g(t)}{\partial t} = - 2\operatorname{Ric}(g(t)).$$
A Ricci soliton is a self-similar solution to the Ricci flow equation since it is obtained
as a rescaling limit of a singularity (\cite{Ham93}, \cite{Sesum06CAG}, \cite{Cao}).

\begin{theorem}[Perelman \cite{Perelman}, see also \cite{derdzinski}]\label{sol1}
Any soliton on a compact manifold is a gradient soliton.
\end{theorem}

\begin{theorem}[Hamilton \cite{Ham93}, Ivey \cite{Ivey}, see also \cite{Cao06}]\label{sol2}
Any nontrivial gradient Ricci soliton on a compact manifold
is shrinking with $\dim M \ge 4$.
\end{theorem}

All known examples of compact Ricci solitons are K\"ahler Ricci solitons:\\
1) Koiso \cite{Koiso90} and Cao \cite{Cao96} : $\bfP^1$-bundles over products of $\bfP^n$'s;\\
2) Wang-Zhu \cite{Wang-Zhu}: toric Fano manifolds; \\
3) Podesta-Spiro \cite{Podesta-Spiro}: homogeneous toric bundles.\\
K\"ahler-Ricci solitons are K\"ahler-Einstein metrics exactly when the invariant in \cite{futaki83.1} vanishes.

For the constructions of solitons on non-compact manifolds, see for example \cite{DancerWang11} and \cite{FutakiWang11}. 

\begin{problem} Does there exist a compact non-K\"ahler nontrivial gradient soliton ?
\end{problem}

The first topic in this survey is about the following lower diameter bound 
for compact gradient shrinking Ricci solitons.

\begin{theorem}[Futaki-Li-Li \cite{FLL}]\label{sol3} Let $M^n$ be a compact manifold with $\dim M = n \ge 4$.
If $g$ is a nontrivial gradient shrinking soliton on $M$ with
$$ 2\operatorname{Ric}(g) - 2\gamma g + \mathcal L_X g = 0$$
then
$$ d_g \ge \frac{2(\sqrt{2} -1)}{\sqrt{\gamma}} \pi$$
where $d_g$ is the diameter of $(M,g)$.
\end{theorem}
As an immediate corollary we have
\begin{corollary}\label{sol4}If a compact gradient shrinking soliton has
\begin{equation}\label{sol5}
d_g < \frac{2(\sqrt{2} -1)}{\sqrt{\gamma}}\pi 
\end{equation}
then $g$ is a trivial soliton (i.e. Einstein).
\end{corollary}
It is interesting to compare Corollary \ref{sol4} with the following theorem of Meyers. 
When $g$ is Einstein with $\operatorname{Ric} = \gamma g$, i.e. $X=0$, and with $\gamma >0$ then
Meyers' theorem says
$$ d_g \le \sqrt{\frac{n-1}{\gamma}} \pi.$$
\begin{problem} : Does there exist an example of Einstein manifold satisfying ${\mathrm (\ref{sol5})}$?
\end{problem}
The proof of Theorem \ref{sol3} is given as follows.
\begin{proof}[Proof of Theorem \ref{sol3}]\ \\
Step 1 (cf. \cite{FS}) :\\ 
Suppose we have a gradient Ricci soliton
$$\operatorname{Ric}(g) - \gamma g + \nabla\nabla f = 0.$$
Define $\Delta_f$ by 
$$ \Delta_f u = \Delta u - \nabla f \cdot \nabla u\ (= g^{ij}\nabla_i\nabla_j u - \nabla^if \nabla_iu).$$
We normalize $f$ so that
$$ \int_M f e^{-f}dV_g = 0.$$
Then $-2\gamma$ is an eigenvalue of $\Delta_f$. In fact
\begin{equation}\label{sol6}
\Delta_f f + 2\gamma f = 0.
\end{equation}

\vspace{3mm}
\noindent
Step 2 (\cite{FLL}) :\\
If $\Delta_f u + \lambda u = 0$ for some nonzero $u \in C^{\infty}(M)$. Then
\begin{equation}\label{sol7}
\lambda \ge \sup_{s \in (0,1)} \{ 4s(1-s)\frac{\pi^2}{d^2} + s \gamma\}
\end{equation}
This Step 2 is the essential part, and its proof is explained later.\\

\vspace{3mm}
\noindent
Step 3: \\
By Step 1 and Step 2 we have for any $s \in (0,1)$ 
$$ 2\gamma \ge 4s(1-s)\frac{\pi^2}{d^2} + s \gamma,$$
and hence
\begin{equation}\label{sol8}
\gamma \ge \frac{4s(1-s)}{2-s}\frac{\pi^2}{d^2}. 
\end{equation}
The right hand side of (\ref{sol8}) takes maximum
\begin{equation}\label{sol9}
 \frac{4s(1-s)}{2-s} \le 12 - 8\sqrt{2}
\end{equation}
at $s = 2 - \sqrt{2} \in (0,1)$. From (\ref{sol8}) and (\ref{sol9}) we get
\begin{eqnarray*}
d_g &\ge& \sqrt{\frac{12 -8\sqrt{2}}{\gamma}}\pi\\
&=& \frac{2(\sqrt{2} -1)}{\sqrt{\gamma}} \pi
\end{eqnarray*}
This completes the proof of Theorem \ref{sol3}. 
\end{proof}
Now we turn to the proof of Step 2 in the proof above. We apply the following result.
\begin{theorem}[\cite{CW1}, \cite{CW2}, \cite{BQ1}, \cite{And-Ni}]\label{sol10}
Let $(M,g)$ be a compact Riemannian manifold, and $\phi$ be a $C^2$ function on $M$. 
Suppose that
$$ {\operatorname Ric}(g) + \nabla\nabla\phi \ge Kg$$
for some $K \in \bfR$. Then the first nonzero eigenvalue $\lambda_1$ for the Witten-Laplacian
$$ \Delta_\phi = \Delta - \nabla^i\phi\cdot\nabla_i$$
satisfies
$$ \lambda_1 \ge \lambda_1(L)$$
where $\lambda_1(L)$ is the first nonzero Neumann eigenvalue of
$$ L = \frac{d^2}{dx^2} - Kx\frac{d}{dx}$$
on $(-d/2, d/2)$.
\end{theorem}
If this theorem is granted, the proof of Step 2 is obtained from the following.
\begin{theorem}[\cite{FLL}]\label{sol11}  Under the notations as above we have
$$ \lambda_1(L) \ge \sup_{s\in (0,1)} \{4s(1-s)\frac{\pi^2}{d^2} + sK\}.$$
\end{theorem}
\begin{proof} We set $\lambda := \lambda_1(L)$ for short.
Let $v$ be the first Neumann eigenfunction for $L$. If we put $D = \frac d2$ and $f = v'$ then $f$ is
the first Dirichlet eigenfunction for $L$ with
\begin{equation}\label{sol12}
f'' - Kxf' = -(\lambda - K)f,
\end{equation}
and
$$ f(-D) = f(D) = 0.$$
Since the first Dirichlet eigenfunction does not change sign we may assume $f > 0$ on $(-D,D)$.
Take any $a > 1$. Then by (\ref{sol12}) we have
\begin{equation}\label{sol13}
\int_{-D}^D f^{a-1}(x)f''(x)dx = -(\lambda -K) \int_{-D}^D f^a(x)dx 
+ \int_{-D}^d Kxf^{a-1}(x)f'(x) dx.
\end{equation}
By integration by parts and $f^{a-1}(\pm D) = 0$ we have
\begin{equation}\label{sol14}
\int_{-D}^D Kxf^{a-1}(x)f'(x)dx = -\frac Ka \int_{-D}^D f^a(x) dx.
\end{equation}
If we put $u = f^{\frac a2}$ then using (\ref{sol13}) and (\ref{sol14}) we get
$$ \frac{4(a-1)}{a^2} \int_{-D}^D |u'|^2 dx = (\lambda - K(1 - \frac 1a))\int_{-D}^D u^2 dx. $$
If we put $s = 1- 1/a$ this is equivalent to
$$ 4s(1-s) \int_{-D}^D |u'|^2 dx = (\lambda - Ks)\int_{-D}^D u^2 dx.$$
Thus we obtain
\begin{eqnarray*}
\frac{\lambda - Ks}{4s(1-s)} &=& \frac{\int_{-D}^D |u'|^2 dx}{\int_{-D}^D u^2 dx}\\
&\ge& \frac {\pi^2}{4D^2} = \frac {\pi^2}{d^2}
\end{eqnarray*}
where the equality holds for $u = \sin(\frac {\pi}{2D} x + \frac {\pi}2)$.
It follows that, for any $s \in (0,1)$ we have
$$ \lambda \ge 4s(1-s)\frac {\pi^2}{d^2} + Ks. $$
This completes the proof of Theorem \ref{sol11}.
\end{proof}
\begin{remark}Theorem 
Theorem \ref{sol3} improves earlier results in \cite{FS} and \cite{And-Ni}. In fact those results follow
from weaker inequalities than (\ref{sol7}). A principle behind those estimates is the 
``Bakry-Emery principle'' \cite{LuRowlettev3} :
If you have to replace the volume form $dV_g$ by $e^{-f}dV_g$
for some reason, then replace $\Delta$ by
$$\Delta_f := \Delta - \nabla f\cdot$$
and replace $\operatorname{Ric}$ by
$$ \operatorname{Ric}_f = \operatorname{Ric} + \nabla\nabla f.$$
Then a theorem for $(dV_g, \Delta, \operatorname{Ric})$ extends to $(e^{-f}dV_g, \Delta_f, \operatorname{Ric}_f)$.
\end{remark}

\section{The mean curvature flow and self-similar \\
solutions on Riemannian cone manifolds.} 

\begin{definition}\label{mcf1} A  Riemannian cone 
manifold $(C(N), \barg)$ over $(N,g)$ consists of a smooth manifold 
$C(N)$ diffeomorphic to $N \times \bfR_+$ and a Riemannian metric 
$$ \barg = dr^2 + r^2 g$$
on $C(N)$ where $r$ is the standard coordinate on $\bfR_+$.
\end{definition}

\begin{definition}\label{mcf2}Let $F : M \to (C(N),\barg)$ be an 
immersion. We call
$$\ArrowF(p)= r(F(p))\frac{\partial}{\partial r}~~\in T_{F(p)}C(N)$$
the {\it position vector} of $F$ at $p \in M$.
\end{definition}

\begin{example}\label{mcf3} If $N$ is $(n-1)$-dimensional standard sphere $S^{n-1}$ we have 
$C(N) = \bfR^n - \{\bfo\}$. For an immersion $F : M \to \bfR^n - \{\bfo\}$ the position vector in the sense
of Definition \ref{mcf2} is 
\begin{equation*}
\ArrowF(p) = r\frac\p{\p r} \in T_{F(p)}\bfR^n - \{\bfo\}.
\end{equation*}
But when $T_{F(p)}\bfR^n - \{\bfo\}$ is identified with $\bfR^n$, $\ArrowF(p)$ is identified with 
$F(p) \in \bfR^n$ which is exactly the position vector in the usual sense. 
\end{example}

An immersion $F : M \to (C(N),\barg)$ is called a 
{\it self-similar solution} to
the mean curvature flow if
$$ H = \lambda \ArrowF^\perp$$
where 
$H$ is the mean curvature vector at $F(p)$, and
$\perp$ denotes the orthogonal projection to the normal bundle.
There are many works for self-similar solutions 
$$F : M \to \bfR^n,$$
e.g. Huisken \cite{H}, Joyce-Lee-Tsui \cite{JoyceLeeTsui}.
We expect that they are extended to
$$F : M \to C(N).$$
As a typical such example we start with the following result which extends a result of 
Huisken \cite{H} in the case when $N = S^{n-1}$ and $C(N) = \bfR^n - \{\bfo\}$.

\begin{theorem}[Huisken]\label{mcf4} Let $M$ be compact manifold and $(N,g)$ a compact Riemannian manifold.
Let $F : M \times [0,T) \to C(N)$ be a mean curvature flow with
$$ \frac{\partial F}{\partial t}(p,t)=H_{t}(p)$$
where $H_t$ is the mean curvature vector of $F_t(M) := F(M,t)$. 
Then the maximal time $T$ of existence of the flow is finite.
\end{theorem}

The following also extends a result of Huisken for $N = S^{n-1}$ and $C(N) = \bfR^n - \{\bfo\}$.
\begin{theorem}[Monotonicity formula, due to Huisken]\label{mcf5}
Let $\rho_T : \bfR \times (-\infty, T) \to \bfR$ be the backward heat kernel
$$ \rho_T(y,t) = \frac 1{(4\pi(T-t))^{m/2}} \exp(- \frac{y^2}{4(T-t)}).$$
Let $F : M \times [0,T) \to C(N)$ be the mean curvature flow. Then
\begin{eqnarray*}
&\frac d{dt}&\int_M \rho_T(r(F_t(p)),t)dV_{g_t} \\
&=& 
-\int_M \rho_{T}(r(F_t(p)),t) 
\biggl\vert\frac{\ArrowF_t^{\bot}(p)}{2(T-t)}
+H_t(p) \biggr\vert_{\barg}^2 dV_{g_t}.
\end{eqnarray*}
\end{theorem}

\begin{definition}[Parabolic rescaling of scale $\lambda$]\label{mcf6}
Given a smooth map $F : M \times [0, T) \to C(N)$, 
the parabolic rescaling $F^{\lambda}:M\times[-\lambda^2 T,0)\rightarrow C(N)$ of scale
$\lambda$ is defined by
\begin{equation*}
F^{\lambda}(p,s)=(\pi_N(F(p,T+\frac{s}{\lambda^2})),\lambda r(F(p,T+\frac{s}{\lambda^2})))
\end{equation*}
where $\pi_N : C(N) = N \times \bfR \to N$ is the standard projection.
\end{definition}

\begin{definition}\label{mcf7} We say that 
a mean curvature flow $F$ develops a type $\mathrm{I}$ singularity
as $t \to T$ if 
there is a constant $C > 0$ such that
$$ \sup_M \left|\mathrm{II}_t\right|^2 \le \frac C{T-t} $$
where $\mathrm{II}_t$ is the second fundamental form of $F_t(M)$.
\end{definition}

Huisken \cite{H} has shown that if a mean curvature flow 
$F : M \times [0,T) \to \bfR^n$ develops a singularity of type $\mathrm{I}$
as $t \to T$ then for any increasing sequence $\{\lambda_i\}_{i=1}^{\infty}$\\
with $\lambda_i \to \infty$ as $i \to \infty$, $\{F^{\lambda_{i}}\}$ subconverges
to a self-similar solution $F^\infty : M_\infty \to \bfR^n$. To get a similar result in the
case of general cone $C(N)$ we need the following definition.
\begin{definition}\label{mcf8} A mean curvature flow $F : M\times [0,T) \to C(N)$ develops 
a type $\mathrm{I}_c$ singularities $t \to T$ if
\begin{enumerate}
\item[(a)]
$F$ develops a type $\mathrm{I}$ singularity as $t \to T$;
\item[(b)]
$r(F_t(p)) \to 0$ for some $p \in M$ as $t \to T$;
\item[(c)]  
For some positive constants $K_1$ and $K_2$ we have 
$K_1(T-t) \leq \min_{ M}r^2(F_{t}) \leq K_2(T-t)$ for all $t \in [0,T)$.
\end{enumerate}
\end{definition}
The following is a cone version of Huisken's result that the limit of parabolic rescalings of a type $\mathrm{I}$ 
singularity yields a self-similar solution. 
\begin{theorem}[\cite{FHY12}]\label{mcf9}
Let $M$ be an $m$-dimensional compact manifold and $C(N)$ the Riemannian cone manifold over an $n$-dimensional Riemannian manifold $(N,g)$. 
Let $F: M\times[0,T)\rightarrow C(N)$ be a mean curvature flow, and assume that $F$ develops a type $\mathrm{I}_c$ singularity at $T$. 
Then, for any increasing sequence $\{\lambda_{i}\}_{i=1}^{\infty}$ of the scales of parabolic rescaling
such that $\lambda_{i}\rightarrow\infty$ as $i\rightarrow\infty$, 
there exist a subsequence $\{\lambda_{i_{k}}\}_{k=1}^{\infty}$ and a sequence $t_{i_k} \to T$ 
such that the sequence of rescaled mean curvature flow $\{F^{\lambda_{i_{k}}}_{s_{i_k}}\}_{k=1}^{\infty}$ 
with $s_{i_k} = \lambda^2_{i_k} (t_{i_k} - T)$ converges to a self-similar solution 
$F^{\infty}:M_\infty \rightarrow C(N)$ to the 
mean curvature flow. 
\end{theorem}

\section{Sasaki manifolds}

In this section we briefly review Sasakian geometry. See for example \cite{futaki10} for more details. 
A Riemannian manifold 
$(N,g)$ with $\dim N = 2m +1$ is called a Sasaki manifold 
if the 
cone manifold 
$(C(N), \overline g)$ with
$C(N) = N \times \bfR_+$, $\overline g = dr^2  + r^2g$, 
is K\"ahler. 
Thus we have 
$\dim_{\bfC} C(N) = m+1$.
We often identify
$N$ with the submanifold $ \{r = 1\}$ in $C(N)$. 

The vector field $\xi = J(r\frac{\partial}{\partial r})$ is called the Reeb vector field.
Then 
$\frac12(\xi - iJ\xi)$ generates a holomorphic flow on $C(N)$.
The 
local leaf spaces of $\frac12(\xi - iJ\xi)$ on $C(N)$ is identified with the local leaf spaces of $\xi$ on $N$.
Here $\xi$ preserves $N \cong \{r=1\}$ and is considered as a vector field on $N$. 
Thus Reeb flow $\mathcal F_{\xi}$ on $N$ 
has a transverse holomorphic structure.

The dual 1-form $\eta$ of $\xi$ is a contact form on $N$. Thus 
$d\eta$ is non-degenerate on $\Ker\ \eta \cong \nu(\mathcal F_{\xi})$
and 
$d\eta/2$ gives $\mathcal F_{\xi}$ a transverse K\"ahler 
structure.
The contact form 
$\eta$ can be lifted to a 1-form $\widetilde{\eta}$ on $C(N)$, and 
$d\widetilde{\eta} = \frac{\sqrt{-1}}2 \partial\barpartial \log r^2$ 
restricts to the transverse K\"ahler form on $N$.
On the other hand, 
$\frac{\sqrt{-1}}2 \partial\barpartial r^2$ is the K\"ahler form on $C(N)$.
A typical Example is the triple 
$$(C(N), N, \mathrm{leaf\ space}) = (\bfC^{m+1}-\{0\}, S^{2m+1}, \bfC\bfP^m).$$
\noindent
Simple curvature computations show that the following proposition.
\begin{proposition}\label{Sasaki1} The following three conditions are equivalent:
\begin{enumerate}
\item $C(N)$ is Ricci-flat K\"ahler;
\item $N$ is Einstein;
\item the local leaf spaces are positive K\"ahler-Einstein.
\end{enumerate}
\end{proposition}
\noindent
Again the triple $(\bfC^{m+1}-\{0\}, S^{2m+1}, \bfC\bfP^m)$ is a typical example of Proposition \ref{Sasaki1}.
\begin{example} Let $M$ be a compact K\"ahler manifold and
$\omega$ a K\"ahler form such that $[\omega]$ is an integral class.
Let $p : L \to M$ be a holomorphic  line bundle with $c_1(L) = - [\omega]$.
We take an Hermitian metric 
$h$ such that $i\partial\barpartial \log h = \omega$ and put 
$r$ to be the distance from the zero section along the fiber with respect to  $h$. 
Let $N \subset L$ be the unit circle bundle (i.e. associated $U(1)$-bundle). 
Then $N$ is a Sasaki manifold, and 
$C(N) = L - \{\mathrm{zero\ section}\}$
where Reeb field generates the $S^1$-action.  
In this case 
$\eta$ is the connection form of $L \to M$.
\end{example}
Let $N$ be a general Sasaki manifold.
A smooth differential form $\alpha$ on $N$ is said to be basic if
$$ i(\xi)\alpha = 0\quad \mathrm{and}\quad \mathcal L_{\xi}\alpha = 0.$$
Denote by $\Omega_B^{p,q}$ the set (or sheaf of germs of) basic $(p,q)$-forms, and then we have natural
operators 
$$ \partial_B : \Omega_B^{p,q} \to \Omega_B^{p+1,q}, \quad\barpartial_B : \Omega_B^{p,q} \to
\Omega_B^{p,q+1}.$$
They satisfy $\partial_B^2 = \barpartial_B^2 = 0$. 
Let $H^{p,q}_B(N)$ be the corresponding cohomology groups which we call the basic cohomology group of type $(p,q)$.
As we can define the Chern classes of compact complex manifolds using Hermitian metrics and its Ricci form, we can define 
Chern classes for transversely holomorphic foliations. They are Chern classes for the normal bundle $\nu(\mathcal F_\xi)$ of the Reeb flow $F_{\xi}$.
They can be expressed as basic cohomology classes, and for this reason they are called the {\it basic Chern classes}. 
In particular we have the first basic first Chern class 
$$c_1^B(\nu(\mathcal F_\xi)) \in H^{1,1}_B(N).$$

Let $\omega^T := \frac 12 d\eta$ be the transverse K\"ahler form.
Suppose $N$ is Sasaki-Einstein. 
Then since $N$ is transversely K\"ahler-Einstein
$$ (2m+2)[\omega^T] = c_1^B(\nu(\mathcal F_\xi)).$$
The following proposition gives a necessary condition for $N$ to be Sasaki-Einstein.
\begin{proposition}\label{Sasaki2} Let $N$ be a compact Sasaki manifold. Then there exists a 
real number $\kappa$ with 
$ c_1^B(\nu(\mathcal F_\xi)) = \kappa [\omega^T] $ for some $\kappa >0$
if and only if 
$c_1(D) = 0$ and $c_1^B > 0$ where $D = \Ker\ \eta$. Here $c_1^B > 0$ means that $c_1^B$ 
is represented by a positive $(1,1)$-form on each local orbit space of the Reeb flow $\mathcal F_\xi$.
\end{proposition}
\begin{definition}We say $g$ is a transverse K\"ahler-Ricci soliton if
$$\rho^T - \omega^T = L_X \omega^T$$
where $\rho^T$ is the Ricci form of $\omega^T$ and $X$ is a 
vector field on $N$ which is obtained as the restriction to $N \cong \{r=1\}$ of the real part of a holomorphic vector field $\tilde X$ on $C(N)$ 
with $[\tilde X, \xi] = 0$.
\end{definition}

It is well-known that, on a Fano manifold, if the Futaki invariant vanishes then K\"ahler-Ricci soliton is a
K\"ahler-Einstein metric. In the similar way, we can define the Futaki invariant, which is called the Sasaki-Futaki invariant, 
and its various extensions for 
transverse K\"ahler structures on compact Sasaki manifolds, and in the case when $c_1(D) = 0$ and $c_1^B > 0$,
if the Sasaki-Futaki invariant vanishes then any transverse K\"ahler-Ricci soliton is a transverse K\"ahler-Einstein metric.
Therefore by Proposition \ref{Sasaki1} we obtain a Sasaki-Einstein metric if we can find a Sasaki manifold with 
$c_1(D) = 0$ and $c_1^B > 0$, with vanishing Sasaki-Futaki invariant and with a transverse K\"ahler-Ricci soliton.
\begin{definition} A Sasaki manifold $N$ is toric if $C(N)$ is toric.
\end{definition}
\begin{theorem} [\cite{FOW}] Let 
$N$ be a compact toric Sasaki manifold with 
$c_1^B > 0$ and $c_1(D) = 0$. 
Then there exists
a transverse K\"ahler-Ricci soliton.
\end{theorem}
\noindent
In the Fano manifold case, this theorem is due to X.-J. Wang and X.-H. Zhu \cite{Wang-Zhu}. 

By leaning the Reeb vector field $\xi$ in the Lie algebra of the torus one can get another Reeb vector field with 
vanishing
Sasaki-Futaki invariant.  This is based on the idea by Martelli-Saprks-Yau \cite{MSY2}, called 
the volume minimization or Z-minimization in AdS-CFT correspondence.  
From this we get the following theorem. 
\begin{theorem}[\cite{FOW}]\label{Sasaki3} Let 
$N$ be a compact toric Sasaki manifold with $c_1^B > 0$ and $c_1(D) = 0$. Then
one can find a deformed Sasaki structure
on which a Sasaki-Einstein metric exists.
\end{theorem}
Now we wish to get a better understanding of the conditions 
for $N$ to have $c_1^B > 0$ and $c_1(D) = 0$. 
\begin{theorem}[\cite{CFO}]\label{Sasaki4} Let $N$ be a compact toric Sasaki manifold with 
$\dim N \ge 5$. We regard $C(N)$ as a toric variety including the apex.
Then the following four conditions are equivalent.
\begin{enumerate}
\item[(a)] $N$ has $c_1^B > 0$ and $c_1(D) = 0$.
\item[(b)] For some positive integer $\ell$, the $\ell$-th power ${\mathcal K}^{\otimes\ell}_{C(N)}$ of the canonical sheaf 
${\mathcal K}_{C(N)}$ is trivial. In particular the apex is a $\bfQ$-Gorenstein
singularity. 
\item[(c)]  The Sasaki manifold $N$ is obtained from a toric diagram with height $\ell$ for some
positive integer $\ell$.
\end{enumerate}
\end{theorem}
Next we wish to explain what the height is. For simplicity we put $n:= m+1$. 
A toric K\"ahler cone of complex dimension $n$ is a K\"ahler cone 
having a Hamiltonian $T^n$-action. Its moment map image is a ``good
rational polyhedral cone'':
$$ C = \{ y \in \mathfrak g^\ast\ |\ \langle y,\lambda_i\rangle \ge 0,\ i=1, \cdots, d\}$$
where $G := T^n$, $\mathfrak g = \mathrm{Lie}(G)$. 
There is an algebraic description of $C$ due to Lerman \cite{Lerman}, but we omit its detail and just define
a good rational polyhedral cone as a moment cone obtained from a toric K\"ahler cone. 
As one of the algebraic properties we have 
$$\lambda_i \in \bfZ_\mathfrak g := \mathrm{Ker}\ \{\exp : \mathfrak g \to G\},$$
that is, every normal vector of facets is in the integral lattice of $\mathfrak g$. 
\begin{definition}[Toric diagram of height $\ell$]\label{slag2}
We say $\{\lambda_i\}_{i=1}^d \subset \bfZ^n \cong \bfZ_{\mathfrak g}$ and 
$\gamma \in \bfQ^n \cong (\bfQ_\mathfrak g)^\ast$ define 
a toric diagram of height $\ell$ if
\begin{enumerate}
\item[(1)]
$ C = \{ y \in \mathfrak g^\ast\ |\ \langle y,\lambda_i\rangle \ge 0,\ i=1, \cdots, d\}$
is a good rational polyhedral cone (i.e. the moment map image of a toric K\"ahler cone).
\item[(2)]
$\ell$ is the largest integer such that $\ell\gamma$ is in $\bfZ^n$ and primitive.
\item[(3)]
$\langle \gamma, \lambda_i \rangle = -1$.
\end{enumerate}
\end{definition}
Using an element of $SL(m+1,\bfZ)$ we can transform $\gamma$ and $\lambda_j$'s so that
$$ \gamma = \left(\begin{array}{r} -\frac1\ell \\ 0 \\ \vdots \\ 0 \end{array}\right)$$
and all $\lambda_j$ are of the form
$$ \lambda_j = \left(\begin{array}{r} \ell \\ \vdots \\ \vdots \end{array}\right).$$
This is the reason why we call ``height $\ell$''.

The height is related to the fundamental group of $N$. 
Let $\mathcal L$ be the subgroup of $\bfZ_{\mathfrak g}$ generated by
$\lambda_1, \cdots, \lambda_d$. Then by Lerman \cite{Lerman02} we have 
$$\pi_1(N) \cong \bfZ_{\mathfrak g}/\mathcal L.$$
Note that
$\bfZ_{\mathfrak g}/\mathcal L$ is not trivial if $\ell > 1$. Thus
if $N$ is a compact connected toric Sasaki manifold 
associated with a toric diagram of height $\ell > 1$, then $N$ is not simply
connected.
The converse is not true as the following example shows. If we take 
$$
\lambda_1 = \left(\begin{array}{c} 1 \\ 0 \\ 0\end{array}\right), 
\lambda_2 = \left(\begin{array}{c} 1 \\ 2 \\ 1\end{array}\right), 
\lambda_3 = \left(\begin{array}{c} 1 \\ 3 \\ 4 \end{array}\right)
$$
then $\pi_1(N) = \bfZ_5$. 

\section{The construction of Special Lagrangian 
submanifolds and Lagrangian self-shrinkers in 
toric Calabi-Yau cones.}

Recall from the previous section that 
a Sasaki manifold $(N,g)$ of dimension $2m+1$ is a toric Sasaki-Einstein manifold 
if 
the cone $(C(N), \barg)$ is a toric Calabi-Yau cone,
i.e. Ricci-flat toric K\"ahler cone of complex dimension $m+1=:n$.
The apex of toric K\"ahler cone $C(N)$ is $\bfQ$-Gorenstein singularity 
if and only if
the moment cone comes from a toric diagram of height $\ell$.
This is equivalent to say that $K_{C(N)}^{\otimes\ell}$ is trivial.
In particular $C(N)$ is Calabi-Yau if $\ell = 1$.
From now on, we shall have toric diagrams of {\it height $1$} in mind even if we use
general $\ell$.
By combining Theorem \ref{Sasaki3} and Theorem \ref{Sasaki4} we obtain the following.
\begin{theorem}[\cite{FOW}, \cite{CFO}]\label{slag4}
For every toric cone manifold $C(N)$ coming from 
toric diagram of height $\ell$, there exists a Ricci-flat
K\"ahler cone metric.
(Equivalently, $N$ admits a Sasaki-Einstein metric.)
\end{theorem}
Suppose $\ell = 1$. Then one can show that 
$$ \Omega = e^{-\sum \gamma_iz^i} dz^1 \wedge \cdots \wedge dz^n $$
is a parallel holomorphic $n$-form, see \cite{CFO} for the proof.
This implies that 
$$ \frac {\omega^n}{n!} = (-1)^{\frac{n(n-1)}2} \binom{\sqrt{-1}}{2}^n \Omega\wedge
\overline{\Omega}$$
for some K\"ahler form $\omega$.
If 
$$ \gamma = \left(\begin{array}{c} -1 \\ 0 \\ \vdots \\ 0 \end{array}\right),$$
then we have 
$$ \Omega = e^{z^1} dz^1 \wedge \cdots \wedge dz^n. $$

The following definitions and theorems  \ref{slag6} -- \ref{slag7} are due to Harvey-Lawson \cite{HarveyLawson82Acta}, and
now are standard in differential geometry.
\begin{definition}\label{slag6}Let $(M,g)$ be a Riemannian manifold. 
A closed k-form $\Omega$ on $(M,g)$ is a calibration
if for any $p\in M$ and for any $k$-dimensional subspace $V \subset T_pM$,
$$ \Omega|V \le dvol_V(g).$$
If a $k$-submanifold $L$ satisfies $\Omega|L = dvol_L(g)$, 
$L$ is called a calibrated submanifold.
\end{definition}
\begin{theorem}
A closed calibrated submanifold minimizes volume in its homology class.
\end{theorem}
\begin{theorem}
Let $(M,\Omega, g)$ be a Calabi-Yau manifold.
Then ${\mathrm Re}\Omega$ defines a calibration. 
\end{theorem}
\begin{definition}\label{slag7}
A calibrated submanifold $L$ for ${\mathrm Re}\Omega$ on a Calabi-Yau manifold 
$(M,\Omega, g)$ is called 
a special Lagrangian submanifold:
$$\dim L = n, \omega|_L = 0, {\mathrm and} {\mathrm Im}\Omega|_L = 0.$$
\end{definition}

\begin{theorem}[\cite{FHY12}]\label{slag5}
With $\ell = 1$, there is a $T^{n-1}$-invariant
special Lagrangian submanifold $L$ described as
$$ L = \mu^{-1}(c) \cap \{(e^{z^1} + (-1)^n e^{\overline{z^1}})/i^n = c'\}$$
where $T^{n-1}$ is generated by $\mathrm{Im}(\p/\p z^2), \cdots, \mathrm{Im}(\p/\p z^n)$
and $\mu : C(N) \to \mathfrak t^{n-1 \ast}$ the moment map. 
\end{theorem}

\begin{example} If $N = S^{2n-1}$ then we have  $C(N) = \bfC^n - \{\bfo\}$. Let $w^1, \cdots, w^n$ be the
holomorphic coordinates on $\bfC^n$. In this case, the conditions in Theorem \ref{slag5} are equivalent to
\begin{enumerate}
\item[(1)]
$|w^j|^2 - |w^1|^2 = c_j,\ j=2, \cdots, n$.
\item[(2)]
If $n$ is even, $\mathrm{Re}(w^1 \cdots w^n) = c'$.\\
If $n$ is odd, $\mathrm{Im}(w^1 \cdots w^n) = c'$.
\end{enumerate}
This special Lagrangian submanifold is due to Harvey-Lawson \cite{HarveyLawson82Acta}.
\end{example}

See also Kawai \cite{Kawai11} for a similar construction.

\begin{theorem}[Yamamoto \cite{Yamamoto}]Let $g$ be any non-negative integer and $\Sigma_g$ be a 
compact surface of genus $g$. Then there exists a
toric Calabi-Yau cone $C(N)$ of height $1$ with $\dim_\bfC C(N) = 3$ such that  
in $C(N)$ 
\begin{enumerate}
\item[(1)]there exist special Lagrangian submanifolds diffeomorphic to $\Sigma_g \times \bfR$ and 
\item[(2)]there exist Lagrangian self-shrinkers diffeomorphic to $\Sigma_g \times S^1$.
\end{enumerate}
\end{theorem}

\section{The diameter of compact self-shrinkers for mean curvature flow}

Let $x:M \to \R^{n+p}$ be an $n$-dimensional submanifold in
the (n+p)-dimensional Euclidean space. If we let the position
vector $x$ evolve in the direction of the mean curvature $\vec{H}$,
then it gives rise to a solution to the mean curvature flow :
$$
x:M\times [0,T)\rightarrow \R^{n+p}, \qquad \frac{\partial
x}{\partial t} = \vec{H}.
$$

We call the immersed manifold $M$ a self-shrinker if it satisfies
the quasilinear elliptic system (see \cite{H}, or \cite{CM}): for some positive constant $\lambda$,
$$
\vec{H}=-\lambda x^{\perp},
$$
where $\perp$ denotes the projection onto the normal bundle of $M$.

We have (see \cite{LW})
$$
\frac 1 {2\lambda} |\vec{H}|^2+\frac{1}{4}\Delta |x|^2=\frac{n}{2}.
$$
Put
$$\phi := 2\lambda(\frac{|x|^2}4 -\frac{n}{4\lambda}).
\eqno (15)
$$
Define the Witten-Laplacian by
$$
\Delta_\phi=\Delta- \nabla\phi\cdot\nabla.
$$
From above formulas, we can check
$$
\begin{array}{lcl}
\Delta_\phi(\frac{1}{4}|x|^2)&=&\Delta (\frac{1}{4}|x|^2)-\frac{\lambda}{8}\nabla|x|^2\cdot \nabla|x|^2\\
&=&\frac{n}{2}- \frac 1{2\lambda}|\vec{H}|^2-\frac{\lambda}2|x^T|^2\\
&=&\frac{n}{2}-\frac{\lambda}{2}|x|^2.
\end{array}
$$
\noindent
Thus we have
$$
\Delta_\phi(\frac{1}{4}|x|^2-\frac{n}{4\lambda})=
-2\lambda(\frac{|x|^2}4-\frac{n}{4\lambda}).
$$

Thus we have proved

\begin{theorem}[\cite{FLL}]\label{th3} In the above situation we have the eigenvalue $2\lambda$ of  the Witten-Laplacian $\Delta_\phi$ with
 eigenfunction $\phi$:
$$
\Delta_\phi\phi=
-2\lambda \phi.
$$
\end{theorem}
Thus we have obtained an eigenvalue similar to (\ref{sol6}), and can hope to get a diameter estimate similar to the case of Ricci solitons.
In fact we can show the following theorem.

\begin{theorem}[\cite{FLL}]\label{th4} Let $x:M\to R^{n+p}$ be an $n$-dimensional compact self-shrinker such that $x(M)$ is not minimal
submanifold in  $S^{n+p-1}(\sqrt{n/\lambda})$, and let
$h^\alpha_{ij}$ be the components of the the second fundamental form
of $M$. Then we have
$$
d\geq \frac{1}{\sqrt{\frac{3\lambda}{2}+\frac{1}{2}K_0}}\pi,
$$
where
$$
K_0:=\max_{1\leq i\leq n}[\sum\limits_{\alpha,k}h^\alpha_{ik}h^\alpha_{ki}].
$$
\end{theorem}

When $p=1$, we have
\begin{corollary}[\cite{FLL}]\label{cor} Let $x:M\to R^{n+1}$ be an $n$-dimensional compact self-shrinker such that $x(M)$ is not
$S^n(\sqrt{n/\lambda})$,  and let $\lambda_i$ be the principal
curvatures of $M$.  Then we have
$$
d\geq \frac{1}{\sqrt{\frac{3\lambda}{2}+\frac{1}{2}K_0}}\pi,
$$
where
$$
K_0:= \max_{p \in M}\max_{1\leq i\leq n}\lambda_i^2.
$$
\end{corollary}

\section{Eternal solutions to K\"ahler-Ricci flow}

Let $M$ be a toric Fano manifold of $\dim M = m$, and
$L \to M$ be a line bundle over $M$ with 
$K_M = L^{-p}$, $p \in \bfZ_+$.
The claim of this section is that we can construct K\"ahler-Ricci solitons on
$L^{-k}$ outside the zero section using Calabi ansatz starting with Sasaki-Einstein metrics on the associated $U(1)$-bundle of $L^{-k}$.

A K\"ahler-Ricci flow is a family  $\omega_t$ of K\"ahler forms satisfying
\begin{equation}\label{KRF1}
\frac d{dt} \omega_t = -  \rho(\omega_t)
\end{equation}
where $\rho(\omega)$ is the Ricci form of $\omega$.
A K\"ahler-Ricci soliton is a K\"ahler form $\omega$ satisfying
\begin{equation}\label{KRF2}
- \rho(\omega) = \lambda \omega + {\mathcal L}_X \omega
\end{equation}
for some holomorphic vector filed $X$ \\
where $\lambda = 1, 0$ or $-1$.
When 
$$ \mathcal L_X \omega = i \partial\barpartial u$$
for some real function $u$, we say that the K\"ahler-Ricci soliton
is a gradient K\"ahler-Ricci soliton.
According as $\lambda = 1, 0$ or $-1$ 
the soliton is said to be expanding, steady and shrinking.

Given a K\"ahler-Ricci soliton with $\lambda = \pm 1$, let $\gamma_t$ be\\
the flow generated by the time dependent vector field
\begin{equation}\label{KRF4}\nonumber
Y_t := \frac 1{\lambda t}X.
\end{equation}
Then
\begin{equation}\label{KRF3}\nonumber
\omega_t := \lambda t \gamma_t^{\ast} \omega 
\end{equation}
is a K\"ahler-Ricci flow.
Notice that, when $\lambda =1$, 
the Ricci flow exists for $t > 0$
and that, when $\lambda =-1$, 
the Ricci flow exists for $t < 0$.
When $\lambda =0$  if we put
\begin{equation}\label{KRF5}\nonumber
\omega_t := \gamma_t^{\ast} \omega 
\end{equation}
where $\gamma_t$ is the flow generated by the vector field $X$,
then $\omega_t $ is a K\"ahler-Ricci flow.
\begin{theorem}[\cite{FutakiWang11}] Suppose $0<k<p$.
\begin{enumerate}
\item[(1)] There exists a shrinking soliton on $L^{-k} - \{\mathrm{zero\ section}\}$.
(The flow exists for $t < 0$.)
\item[(2)] There exists an expanding soliton on $L^{-k} - \{\mathrm{zero\ section}\}$.
(The Ricci flow exists for $t > 0$.)
\item[(3)] The first one and the second one can be pasted to form an eternal solution of the Ricci flow on
$(L^{-k} - \{\mathrm{zero\ section}\}) \times (-\infty, \infty)$.
\end{enumerate}
\end{theorem}
\begin{remark}\label{eternal1}For $p < k$, there exists an expanding soliton on $(L^{-k} - \{\mathrm{zero\ section}\})$.
\end{remark}
For $M = \bfC\bfP^m$, $L^{-1} = \mathcal O(-1)$, 
(2) is due to H.-D. Cao, and extends to $\mathbb C^{m+1}$, and
(1) and Remark \ref{eternal1} are due to Feldman, Ilmanen and Knopf \cite{FIK}.
The solutions in (1) and Remark \ref{eternal1} due to \cite{FIK} extend to the zero section.
So, the solution exists on 
$$L^{-1}\times (-\infty,0)\ \cup\ ({\mathbb C}^{m+1}-\{{\bf o}\}) \times \{0\} 
\ \cup\ \mathbb C^{m+1} \times (0,\infty).$$

Now we explain the 
Calabi ansatz (momentum construction) in the following classical case. 
Let $(M,\omega)$ be a Fano K\"ahler-Einstein manifold, and 
search for a Ricci-flat K\"ahler metric on $K_M$.

Let $p : K_M \to M$ be the canonical line bundle, 
$h$ an Hermitian metric of $K_M$ such that 
$$ i\partial\barpartial \log h = \omega. $$
Define $r : K_M \to \bfR$ by $r(z) = \sqrt{h(z,z)}$.
Search for a K\"ahler metric of constant scalar curvature 
of the form $\tilde{\omega} = p^{\ast}\omega + i \partial\barpartial f(r)$ 
where $f(r)$ is a smooth function of $r$.
We obtain a 2nd order ODE in terms of 
$\varphi(\sigma) := f^{\prime\prime}(r)$
where $\sigma = f^\prime(r)$.
Although the equation was set up to find a constant scalar curvature K\"ahler metric, 
the metric obtained happens to be Ricci-flat. Hence we have obtained a Ricci-flat K\"ahler metric
except the zero section.
Next task is to find conditions so that the metric extends smoothly to the zero section.
The answer is the following: 
$$\varphi(0) = 0\ \quad \mathrm{and}\ \quad \varphi'(0) = 2.$$ 
Another task is to 
find conditions so that the metric becomes complete near $r = \infty$.
The answer in this case is the following: 
$$\varphi(r) = O(r^2)\ \quad \mathrm{as}\ \quad r \to \infty.$$
Instead of K\"ahler-Einstein manifolds, we use Sasaki-Einstein manifolds.
The same idea applies for the construction of gradient K\"ahler-Ricci solitons except the extension to
the zero section. The difficulty of the extension to the zero section arises when the Reeb vector field is
irregular. 
As a version of  Theorem \ref{slag4} we have the following. 
Suppose
\begin{enumerate}
\item[(i)]Suppose that $M$ is toric Fano, and take the K\"ahler class $[\omega] = c_1(M)$,
\item[(ii)]consider the Sasaki manifold $N \subset K_M$ to be the associated $U(1)$-bundle.
\end{enumerate}
Then $N$ admits a Sasaki-Einstein metric possibly with irregular Reeb vector field.

The Calabi ansatz for Sasakian manifold $N$ is described as 
$$p : C(N) = K_M - \{\mathrm{zero\ section}\} \to N,$$
$$\widetilde{\omega} = p^{\ast}\left(\frac12 d\eta\right) + i \partial \barpartial f(r).$$
Summary of this section is :
\begin{enumerate}
\item[1.] If $M$ is a toric Fano manifold then $U(1)$-bundle $N$ associated with $K_M$ is a 
toric Sasakian manifold.
\item[2.] By Futaki-Ono-Wang \cite{FOW} $N$ admits a possibly irregular Sasaki-Einstein metric.
\item[3.] We can apply Calabi's ansatz to get K\"ahler-Ricci solitons in $L^k$ outside the zero section for $L$ with $L^{-p} = K_M$
and $0<k<p$.
\end{enumerate}

\bibliographystyle{amsalpha}

\begin{thebibliography}{A}


\bibitem{And-Ni}B.~Andrew and L. Ni, Eigenvalue comparison on Bakry-Emery manifolds, arXiv:1111.4967.

\bibitem{BQ1} D. Bakry, Z.-M. Qian, Some new results on eigenvectors via dimension, diameter, and Ricci curvature, Adv. in Math. 155 (2000), 98-153.

\bibitem{Cao96}H.-D.~Cao, Existence of gradient K\"ahler-Ricci solitons, Elliptic and Parabolic Methods in Geometry
(Minneapolis, MN, 1994), A. K. Peters (ed.), Wellesley, MA, 1996, 1-16.

\bibitem{Cao06}H.-D.~Cao, Geometry of Ricci solitons. Chinese Ann. Math. Ser. B 27 (2006), no. 2, 121--142. 

\bibitem{Cao} H.-D. Cao, Recent progress on Ricci solitons. Recent advances in geometric analysis, 1--38, 
Adv. Lect. Math. (ALM), 11, Int. Press, Somerville, MA, 2010.



\bibitem{CW1} M.-F. Chen, F.-Y. Wang, Application of coupling method to the first eigenvalue on manifolds, 
Sci. Sinica, (A) 37 (1994), 1-14.

\bibitem{CW2} M.-F. Chen, F.-Y. Wang, General formula for lower bound of the first eigenvalue on Riemannian manifolds, 
Sci. Sinica (A), 40 (1997), 384-394.


\bibitem{CFO} K.~Cho, A.~Futaki and H.~Ono : 
Uniqueness and examples of compact toric Sasaki-Einstein metrics, 
 Comm. Math. Phys., 277 (2008), 439-458. math.DG/0701122


\bibitem{CM} T. H. Colding, and W.P. Minicozzi II, Generic mean curvature flow
I: generic singularities, Ann. math., 175(2) (2012), 755-833.

\bibitem{DancerWang11}A.S.~Dancer and M.Y.~Wang :
On Ricci solitons of cohomogeneity one, Ann. Global Anal. Geom. 39 (2011), no. 3, 259--292.

\bibitem{derdzinski}A.~Derdzinski, Compact Ricci solitons, preprint.

\bibitem{FIK} M.~Feldman, T.~Ilmanen and D.~Knopf : Rotationally symmetric shrinking and expanding 
gradient K\"ahler-Ricci solitons, J. Differential Geometry, 65(2003), 169-209.


\bibitem{futaki83.1}A.~Futaki : 
An obstruction to the existence of Einstein K\"ahler metrics, Invent. 
Math. {\bf 73}, 437-443 (1983).

\bibitem{futaki10}A.~Futaki : 
Toric Sasaki-Einstein Geometry, Fourth International Congress of Chinese Mathematicians (eds. L. Ji et al), 
AMS/IP Studies in Advanced Mathematics, Vol.48(2010), 107-125.

\bibitem{FHY12}A.~Futaki, K.~Hattori and H.~Yamamoto : Self-similar solutions to the mean curvature
flows on Riemannian cone manifolds and special Lagrangians on toric Calabi-Yau cones, arXiv:1112.5933.


\bibitem{FLL}A.~Futaki, H. Li and X-D. Li, On the first eigenvalue of the Witten-Laplacian and the diameter of 
compact shrinking solitons, arXiv:1111.6364.

\bibitem{FOW}A.~Futaki,  H.~Ono and G.~Wang : 
 Transverse K\"ahler geometry of Sasaki manifolds and toric Sasaki-Einstein manifolds,
 J. Differential Geom., 83(2009), 585-636.

\bibitem{FS} A. Futaki, Y. Sano, Lower diameter bounds for compact shrinking Ricci solitons, to appear in Asian J. Math., arXiv:1007.1759v1.

\bibitem{FutakiWang11}A.~Futaki and M.T.~Wang : Constructing K\"ahler-Ricci solitons from Sasaki-Einstein manifolds,  
Asian J. Math. 15 (2011), no. 1, 33--52.

\bibitem{Ham93}R.S.~ Hamilton,  The formation of singularities in the Ricci flow, Surveys in Differential Geometry 
(Cambridge, MA, 1993), 2, International Press, Combridge, MA, 1995, 7-136.

\bibitem{HarveyLawson82Acta}
R.~Harvey and H.B.~Lawson Jr. : Calibrated geometries, Acta Math., 148 (1982), 47--157. 


\bibitem{H} G. Huisken, Asymptotic behavior for singularities of the mean
curvature flow, J. Differential Geom., {\bf 31} (1990), no. 1, 285-299.


\bibitem{Ivey} T.~Ivey, Ricci solitons on compact three-manifolds. Differential Geom. Appl. 3 (1993), no. 4, 301--307.
 
\bibitem{JoyceLeeTsui}D.~Joyce, Y.-I.~Lee, M.-P.~Tsui : Self-similar solutions and translating solitons for Lagrangian mean curvature flow, 
J. Differential Geom. 84 (2010), no. 1, 127--161. 

\bibitem{Kawai11}K.~Kawai : Torus invariant special Lagrangian submanifolds in the canonical bundle of toric 
positive K\"ahler Einstein manifolds, Kodai Math. J., 34(2011), 519--535.



\bibitem{Koiso90}N.~Koiso, On rotationally symmetric Hamilton's equation for K\"ahler-Einstein metrics. 
Recent topics in differential and analytic geometry, 327--337, 
Adv. Stud. Pure Math., 18-{\rm I}, Academic Press, Boston, MA, 1990. 

\bibitem{Lerman} E.~Lerman : 
Contact toric manifolds,
 J. Symplectic Geom.  1  (2003),  no. 4, 785--828.
 
 \bibitem{Lerman02} E.~Lerman : 
Homotopy groups of K-contact toric manifolds,
 Transact. Amer. Math. Soc., 356 (2004),  no. 4, 4075--4084.
 
\bibitem{LW} H. Li and Y. Wei, Lower volume growth estimates for self-shrinkers of mean curvature flow, arXiv: 1112.0828v3.


\bibitem{LuRowlettev3}Z. Lu, J. Rowlett, Eigenvalues of collapsing domains and drift Laplacians, arXiv:1003.0191v3.

\bibitem{MSY2} D.~Martelli, J.~ Sparks and  S.-T. ~Yau : Sasaki-Einstein manifolds and volume minimisation, 
Comm. Math. Phys. 280 (2008), no. 3, 611--673.


\bibitem{Perelman}G.~Perelman : The entropy formula for the Ricci flow and its geometric applications, 
http://arXiv.org/abs/math.DG/02111159, 2002.

\bibitem{Podesta-Spiro}
F.~Podesta and A.~Spiro : K\"ahler-Ricci solitons on homogeneous toric bundles. J. Reine Angew. Math. 642 (2010), 109--127. 

\bibitem{Sesum06CAG} N.~Sesum : Convergence of the Ricci flow toward a soliton, Comm. Anal. Geom. 14 (2006), no. 2, 283--343.

\bibitem{Wang-Zhu} X.-J. Wang and X.-H. Zhu : 
K\"ahler-Ricci solitons on toric manifolds with positive first Chern class,
Adv. Math.  188  (2004),  no. 1, 87--103.

\bibitem{Yamamoto}H.~Yamamoto : Special Lagrangians and Lagrangian self-similar solutions in cones over toric Sasaki manifolds, preprint.
arXiv:1203.3934.

%%%%%%%%%%%%%%%%%%%%%%%%%%%%%%%%%%%%%%%%%%%%%%%%%%%%%


\end{thebibliography}

\end{document}